\documentclass{amsart}
\usepackage{amsmath,amssymb}
\usepackage{amsthm}
\usepackage[alphabetic]{amsrefs}

\theoremstyle{definition}
\newtheorem{definition}{Definition}[section]
\theoremstyle{definition}

\theoremstyle{definition}
\newtheorem{example}[definition]{Example}
\theoremstyle{definition}
\newtheorem{notation}[definition]{Notation}
\theoremstyle{definition}

\theoremstyle{remark}
\newtheorem*{remark}{Remark}
\theoremstyle{plain}
\newtheorem{theorem}[definition]{Theorem}
\theoremstyle{plain}
\newtheorem{stheorem}[definition]{Structure Theorem}
\theoremstyle{plain}
\newtheorem{lemma}[definition]{Lemma}
\theoremstyle{plain}
\newtheorem{proposition}[definition]{Proposition}
\theoremstyle{plain}

\theoremstyle{plain}
\newtheorem{corollary}[definition]{Corollary}

\newcommand{\M}[2]{\mathcal{M}_{#1}^{#2}}
\newcommand{\I}[3]{\mathcal{I}_{#1}^{#2}(#3)}
\newcommand{\Q}[3]{\mathcal{Q}_{#1}^{#2}(#3)}
\newcommand{\iI}[3]{\overline{\mathcal{I}}_{#1}^{#2}(#3)}
\newcommand{\iQ}[3]{\overline{\mathcal{Q}}_{#1}^{#2}(#3)}
\newcommand{\Sc}[3]{\mathcal{S}_{#1}^{#2}(#3)}
\newcommand{\Tc}[3]{\mathcal{T}_{#1}^{#2}(#3)}
\newcommand{\Z}{\mathbb{Z}}
\newcommand{\Zo}{\mathbb{Z}_{\ge0}}

\newcommand{\lexl}{\succ_{\text{lex}}}
\newcommand{\lexleq}{\succeq_{\text{lex}}}
\newcommand{\ith}[1]{#1^{\text{th}}}

\DeclareMathOperator{\im}{im}
\DeclareMathOperator{\ct}{ct}
\DeclareMathOperator{\ft}{ft}

\begin{document}

\title{A Duality between Hilbert Functions of Lex Ideals and Quotients}

\author{Reid Buchanan}
\address{Department of Mathematics, Oklahoma State University, 401 Mathematical Sciences, Stillwater, OK 74078}
\email{rebucha@okstate.edu}

\thanks{\textbf{Acknowledgements:} The author would like to thank Alessandra Costantini, Hoai Dao, Christopher Francisco, Jeffrey Mermin, Anand Patel, and Jay Schweig for several helpful conversations regarding the content of the paper.}

\begin{abstract}
We study the Macaulay coefficients induced by the ideal and quotient segments of a degree-$\delta$ monomial in $n$ variables.  We give explicit formulas for these coefficients and establish a duality between the two theories.  Our main result is that the ideal and quotient coefficients form a set partition of $\{0,1,\ldots,n+\delta-2\}$.
\end{abstract}

\maketitle

\section{Introduction}
Let $S=k[x_1,x_2,\ldots,x_n]$ and $M$ be a graded module over $S$.  Of particular interest is the rate in which the Hilbert function of $M$ grows (the Hilbert function of $M$ encodes the degree $\delta$ components: $\operatorname{Hilb}_M(\delta)=\dim M_\delta$).  Lex ideals allow us to give a bound for this growth that may be expressed in terms of binomial coefficients.  In particular, if $L$ is a lex ideal such that $\dim(S/L)_\delta=\dim(S/M)_\delta$ for some degree $\delta$, then $\dim(S/S_1L)_{\delta+1}\ge\dim(S/M)_{\delta+1}$.  There is a dual result for ideals: if $\dim L_\delta=\dim M_\delta$, then $\dim S_1L_\delta\le\dim M_{\delta+1}$.  The expressions $\dim(S/S_1L)_{\delta+1}$ and $\dim S_1L_\delta$ can be explicitly computed via a decomposition of the defining spaces into monomial spaces.  For any positive integers $s$ and $p$, a standard combinatorial result guarantees that $s$ has a unique Macaulay representation

\[s=\sum_{i=1}^p\binom{s_i}{i}\]

\noindent for some unique integers $s_1<s_2<\cdots<s_p$.  The Macaulay representation bounds the growth of the Hilbert function.

\begin{theorem}\label{t:introquot}
    Let $I$ be a graded ideal and fix $\delta$.  Write the $\ith{\delta}$ Macaulay representation of $\dim(S/I)_\delta$

    \[\dim(S/I)_\delta=\sum_{i=1}^{\delta}\binom{t_i}{i}\]

    \noindent as above.  Then

    \[\dim(S/I)_{\delta+1}\le\sum_{i=1}^{\delta}\binom{t_i+1}{i+1}.\]
\end{theorem}

\begin{theorem}\label{t:introideal}
    Let $I$ be a graded ideal and fix $n$.  Write the $\ith{(n-1)}$ Macaulay representation of $\dim I_\delta$

    \[\dim I_\delta=\sum_{i=1}^{n-1}\binom{s_i}{i}\]

    \noindent as above.  Then

    \[\dim I_{\delta+1}\ge\sum_{i=1}^{n-1}\binom{s_i+1}{i}.\]
\end{theorem}

The proof for these bounds uses Macaulay's theorem, which states that lex ideals in $S$ attain all possible Hilbert functions of ideals in $S$. \cite{Macaulay}  For a proof of Macaulay's theorem, see \cite{Peeva}*{Chapter 45}.  Analogues of Macaulay's theorem exist for more general settings (see \cite{Kruskal}, \cite{Katona}, \cite{ClementsLindstrom}, \cite{CooperRoberts}, and \cite{CavigliaKummini}), whose lex ideals have similar combinatorial characteristics of Hilbert function growth.

There is an apparent duality between Theorems \ref{t:introquot} and \ref{t:introideal}: the setting of the quotient depends on the degree and the ideal setting instead depends on the number of variables.  We explore this respective dependency in more detail and bring to attention a number of parallel properties both settings have.  The proofs of the combinatorial bounds in Theorem \ref{t:introquot} and (especially) Theorem \ref{t:introideal} are largely attributed to folklore.  We give efficient proofs of both, as well as a way to determine their Macaulay coefficients directly from a monomial defining the corresponding lex ideal $L$ rather than having to compute the numerical value of $\operatorname{Hilb}_L(\delta)$.

Throughout, fix positive integers $n,\delta$ and a monomial $m$.  The ideal segment of such a monomial $m$ is the vector space spanned by the monomials preceding $m$ in the lex order, while the quotient segment is the corresponding space spanned by those following $m$.  Given $m$, we decompose its ideal and quotient segments, which allows us to construct the Macaulay representations of the segments' dimensions.  Along the way, we obtain explicit expressions for the complete decompositions based on $m$ (Structure Theorem \ref{st:strongdecomp}).  For both segments, we also give explicit formulas for the Macaulay coefficients in terms of given information directly from $m$.  The main result, Theorem \ref{t:setpart}, shows that the Macaulay coefficients for the ideal and quotient segments form a set partition of $\{0,1,\ldots,n+\delta-2\}$.

The paper is structured as follows.  In Section 2, we establish the necessary notation and background content.  In Section 3, we decompose lex ideals and quotients as direct sums of simpler monomial spaces and use this decomposition to show how the lex ideals and quotients grow.  In Section 4, we show how to read off the Macaulay coefficients from a generating monomial and use this characterization to prove the main result.

\section{Background and Notation}

\begin{notation}\label{n:monosp}
Let $\M{n}{\delta}$ be the space generated by all degree-$\delta$ monomials in $n$ variables.  We shall also let $\M{[r,s]}{\delta}$ be the set generated by all monomials in the variables $x_r,x_{r+1},\ldots,x_s$, which will be used in some of the proofs.  We note the natural identification between $\M{[r,s]}{\delta}$ and $\M{s-r+1}{\delta}$, which will be understood without comment in these proofs.
\end{notation}

\begin{notation}
We may write any degree-$\delta$ monomial $m=x_1^{\alpha_1}\cdots x_n^{\alpha_n}$ as $m=x_{j_1}x_{j_2}\cdots x_{j_\delta}$, for indices $1\le j_1\le j_2\le\cdots\le j_\delta\le n$.  We call this the standard factorization of $m$.
\end{notation}

\begin{example}
If $m=x_2^3x_4x_5^2$, then $m=x_2x_2x_2x_4x_5x_5$, so that $j_1=j_2=j_3=2$, $j_4=4$, and $j_5=j_6=5$.
\end{example}

\begin{definition}
Let $m=x_1^{\alpha_1}\cdots x_n^{\alpha_n}=x_{j_1}\cdots x_{j_\delta}$ be a monomial.  Define $\min(m):=\min(\{i\ |\ x_i\text{ divides }m\})=j_1$.  We will call this the \emph{minimum of} $m$.

Likewise, define the \emph{maximum of $m$} to be $\max(m):=\max(\{i\ |\ x_i\text{ divides }m\})=j_\delta$.  Our focus will require the minimum operation much more than the maximum.
\end{definition}

\begin{definition}
For monomials $m=x_1^{\alpha_1}x_2^{\alpha_2}\cdots x_n^{\alpha_n}$ and $m'=x_1^{\beta_1}x_2^{\beta_2}\cdots x_n^{\beta_n}$, we say that $m\lexl m'$ if and only if $\alpha_i>\beta_i$ for the smallest value of $i$ such that $\alpha_i\ne\beta_i$.  This defines a monomial order in the indeterminants $x_1,\ldots,x_n$.  We call this the \emph{lexicographical} (or \emph{lex}) \emph{order} on the set of monomials.  If $m\lexl m'$, we say that $m$ is \emph{lex-larger than} $m'$ and $m'$ is \emph{lex-smaller than} $m$.
\end{definition}

\begin{example}
If $S=k[a,b,c]$, then the ordering of all degree-$3$ monomials lexicographically is as follows:
\[a^3\lexl a^2b\lexl a^2c\lexl ab^2\lexl abc\lexl ac^2\lexl b^3\lexl b^2c\lexl bc^2\lexl c^3.\]
\end{example}

\begin{remark}
Although it is counterintuitive that $x_i\lexl x_j$ when $i<j$, this notation is standard in the literature.
\end{remark}
  
\begin{definition}
Let $m$ be a degree-$\delta$ monomial in $n$ variables.  Define $\I{n}{\delta}{m}:=\left(\mu\in\M{n}{\delta}\ \middle|\ \mu\lexl m\right)$, where the inequality is strict.  We will call this the \emph{(exclusive) ideal segment of $m$}.  Similarly, let $\iI{n}{\delta}{m}:=\left(\mu\in\M{n}{\delta}\ \middle|\ \mu\lexleq m\right)$ be the \emph{inclusive ideal segment of $m$}.  The unmodified phrase ``ideal segment'' will refer to the exclusive segment.  We shall also adopt the convention introduced in Notation \ref{n:monosp} in proofs: $\I{[r,s]}{\delta}{m}:=\left(\mu\in\M{[r,s]}{\delta}\ \middle|\ \mu\lexl m\right)$ and its counterpart $\iI{[r,s]}{\delta}{m}:=\left(\mu\in\M{[r,s]}{\delta}\ \middle|\ \mu\lexleq m\right)$.
\end{definition}

\begin{definition}
An \emph{(initial) lex segment} is an inclusive or exclusive ideal segment of some monomial.  A \emph{lex ideal} is a graded ideal whose degree-$\delta$ components are lex segments.
\end{definition}

\begin{definition}
Suppose $m\in\M{n}{\delta}$ such that $m\ne x_1^{\delta}$.  The \emph{predecessor of $m$} is the lex-smallest monomial generator in $\I{n}{\delta}{m}$.  That is, if $\mu$ is the predecessor of $m$, then $\I{n}{\delta}{m}=\iI{n}{\delta}{\mu}$.
\end{definition}

\begin{lemma}\label{l:monopred}
Let $m=x_1^{\alpha_1}\cdots x_n^{\alpha_n}\in\M{n}{\delta}$ and suppose $m\ne x_1^\delta$.  Write $m=wx_{\max(m)}^\gamma$, where $\gamma=\alpha_{\max(m)}$ and $w\in\M{[1,\max(m)-1]}{\delta-\gamma}$.  Then the predecessor of $m$ is $wx_{\max(m)-1}x_n^{\gamma-1}$.
\end{lemma}

\begin{proposition}\label{p:idealmult}
Write $S=\bigoplus_iS_i$ and let $\mathcal{I}\subseteq\M{n}{\delta}$ be a lex segment.  Then $S_{1}\mathcal{I}\subseteq\M{n}{\delta+1}$ is also a lex segment.  More precisely, $S_{1}\I{n}{\delta}{m}=\I{n}{\delta+1}{mx_{\max(m)}}$ and $S_{1}\iI{n}{\delta}{m}=\iI{n}{\delta+1}{mx_n}$.
\end{proposition}
\begin{proof}
First, we'll prove $S_{1}\iI{n}{\delta}{m}=\iI{n}{\delta+1}{mx_n}$.  The lex-smallest monomial generator of $S_{1}\iI{n}{\delta}{m}$ is $mx_n$.  It suffices now to prove that this is a lex segment.  Let $m'x_{\max(m')}\lexleq mx_n$ for some $m'\in\M{n}{\delta}$.

We claim $m'\lexleq m$.  Suppose $m'\prec_{\text{lex}}m$.  Then $m'x_{\max(m')}\prec_{\text{lex}}mx_{\max(m')}\preceq_{\text{lex}}mx_n$, which proves the claim.

Thus, $m'\in\iI{n}{\delta}{m}$, so $m'x_{\max(m')}\in S_1\iI{n}{\delta}{m}$.  Therefore, $S_1\iI{n}{\delta}{m}=\iI{n}{\delta+1}{mx_n}$.

Now we will prove $S_1\I{n}{\delta}{m}=\I{n}{\delta+1}{mx_{\max(m)}}$.  Suppose $\mu$ is the predecessor of $m$.  Then $S_1\I{n}{\delta}{m}=S_1\iI{n}{\delta}{\mu}=\iI{n}{\delta+1}{\mu x_n}$.  It suffices to prove that $\mu x_n$ is the predecessor of $mx_{\max(m)}$.  Write $m=wx_{\max(m)}^\gamma$, with $\gamma$ maximal, so that $\mu=wx_{\max(m)-1}x_n^{\gamma-1}$ by Lemma \ref{l:monopred}.  Then $mx_{\max(m)}=wx_{\max(m)}^{\gamma+1}$ and $\mu x_n=wx_{\max(m)-1}x_n^\gamma$, which is the predecessor of $mx_{\max(m)}$ as desired.  Therefore, $S_1\I{n}{\delta}{m}=\iI{n}{\delta+1}{\mu x_n}=\I{n}{\delta+1}{mx_{\max(m)}}$.
\end{proof}

Macaulay's theorem motivates the study of lex ideals by classifying all possible Hilbert functions.

\begin{definition}
Let $M=\bigoplus_{i\in\Z}M_i$ be a graded module.  The \emph{Hilbert function of $M$}, $\operatorname{Hilb}_M:\Z\to\Zo$, is given by $\operatorname{Hilb}_M(i)=\dim M_i$.
\end{definition}

\begin{theorem}[Macaulay's theorem]
Let $\mathcal{J}$ be a graded ideal of $S$.  Then there exists a lex ideal $\mathcal{I}$ of $S$ such that $\mathcal{I}$ and $\mathcal{J}$ have the same Hilbert function.
\end{theorem}

There are several approaches to prove Macaulay's theorem, some of which are found in \cite{Peeva}*{Chapter 45}.

Macaulay's theorem holds for quotients by graded ideals as well, allowing us to study the Hilbert functions of quotients in a very similar way.  The theory of quotient segments is dual to the theory of ideal segments.

\begin{notation}
Let $m$ be a degree-$\delta$ monomial in $n$ variables.  Define $\Q{n}{\delta}{m}:=\left(\mu\in\M{n}{\delta}\ \middle|\ m\lexl\mu\right)$, where the inequality is strict.  We will call this the \emph{(exclusive) quotient segment of $m$}.  Similarly, let $\iQ{n}{\delta}{m}:=\left(\mu\in\M{n}{\delta}\ \middle|\ m\lexleq\mu\right)$ be the \emph{inclusive quotient segment of $m$}.  The unmodified phrase ``quotient segment'' will refer to the exclusive segment.  We shall also adopt the convention introduced in Notation \ref{n:monosp} in proofs: $\Q{[r,s]}{\delta}{m}:=\left(\mu\in\M{[r,s]}{\delta}\ \middle|\ m\lexl\mu\right)$ and its counterpart $\iQ{[r,s]}{\delta}{m}:=\left(\mu\in\M{[r,s]}{\delta}\ \middle|\ m\lexleq\mu\right)$.
\end{notation}

\begin{lemma}\label{l:reducequoseg}
Let $m\in\M{[i,n]}{\delta}$ where $i<n$ is an index such that $x_i\nmid m$.  Then $\Q{[i,n]}{\delta}{m}=\Q{[i+1,n]}{\delta}{m}$.
\end{lemma}

\begin{definition}
A \emph{terminal lex segment} is an inclusive or exclusive quotient segment of some monomial.
\end{definition}

We also get a counterpart to Proposition \ref{p:idealmult}:

\begin{proposition}\label{p:quotientmult}
Suppose $\mathcal{I}\subseteq\M{n}{\delta}$ is a lex segment.  Then, setting $L$ equal to the ideal generated by $\mathcal{I}$, we have:
  \begin{enumerate}
  \item If $(S/L)_\delta=\Q{n}{\delta}{m}$, then $(S/S_1L)_{\delta+1}=\Q{n}{\delta+1}{mx_{n}}$.
    \item If $(S/L)_\delta=\iQ{n}{\delta}{m}$, then  $(S/S_1L)_{\delta+1}=\iQ{n}{\delta+1}{mx_{\max(m)}}$.
  \end{enumerate}
\end{proposition}
\begin{proof}\leavevmode
\begin{enumerate}
\item[(1)] We have $L_\delta=\iI{n}{\delta}{m}$, so $S_1L_\delta=\iI{n}{\delta+1}{mx_n}$ by Proposition \ref{p:idealmult}.  Thus, $(S/S_1L)_{\delta+1}=\Q{n}{\delta+1}{mx_n}$.

\item[(2)] We have $L_\delta=\I{n}{\delta}{m}$, so $S_1L_\delta=\I{n}{\delta+1}{mx_{\max(m)}}$ by Proposition \ref{p:idealmult}.  Thus, $(S/S_1L)_{\delta+1}=\iQ{n}{\delta+1}{mx_{\max(m)}}$.\qedhere
\end{enumerate}
\end{proof}

Finally, we recall the Macaulay representation of a number:

\begin{theorem}\label{t:macrep}
  Let $s$ be a nonnegative integer and $p$ a positive integer.  Then there exists a unique decreasing sequence of nonnegative integers $s_p>s_{p-1}>\dots>s_1$ such that
  \[s=\binom{s_p}{p}+\binom{s_{p-1}}{p-1}+\dots+\binom{s_1}{1}.\]
  (We use the convention that $\binom{i}{j}=0$ whenever $i<j$.)
\end{theorem}

\begin{definition} Let $s$ and $p$ be as in Theorem \ref{t:macrep}.  Then the expression 
  \[s=\binom{s_p}{p}+\binom{s_{p-1}}{p-1}+\cdots+\binom{s_1}{1}\]
  is called the \emph{$\ith{p}$ Macaulay representation} of $s$, and the numerators $(s_p,s_{p-1},\ldots,s_1)$ are called its \emph{($\ith{p}$) Macaulay coefficients}.
\end{definition}

\begin{example}
The $\ith{6}$ Macaulay coefficients of $114$ are $(9,7,5,4,1,0)$, since we have
\[114=\binom{9}{6}+\binom{7}{5}+\binom{5}{4}+\binom{4}{3}+\binom{1}{2}+\binom{0}{1}.\]
\end{example}

\begin{proof}[Proof of Theorem \ref{t:macrep}]
  We defer a careful proof that the Macaulay coefficients exist until the next section, where we will construct them in the context of Hilbert functions and lex ideals.  However, we note that the greedy algorithm works:  Choose $s_p$ maximal such that $\binom{s_p}{p}\leq s$, and then proceed by induction.

  To prove the sequence $(s_p,\ldots,s_1)$ is unique, we will induct on $p$.  Suppose for all nonnegative integers less than $s$ the corresponding $(p-1)$-tuple is unique.  We will prove for any pair of sums

  \[s=\sum_{i=1}^p\binom{s_i}{i}=\sum_{i=1}^p\binom{s_i'}{i}\]

  \noindent that $s_p=s_p'$.  Suppose without loss of generality that $s_p<s_p'$ for a contradiction.  Then

  \begin{align*}
  &s=\sum_{i=1}^p\binom{s_i}{i}\le\sum_{i=1}^p\binom{s_p-p+i}{i}<\sum_{i=0}^p\binom{s_p-p+i}{i} \\
  &\qquad=\binom{s_p+1}{p}\le\binom{s_p'}{p}\le\sum_{i=1}^p\binom{s_i'}{i}=s.
  \end{align*}

  \noindent This is a contradiction, so $s_p=s_p'$, and since
  
\[s-\binom{s_p}{p}=\sum_{i=1}^{p-1}\binom{s_i}{i}=\sum_{i=1}^{p-1}\binom{s_i'}{i},\]

  \noindent each $s_i=s_i'$ for all $1\le i\le p-1$ by the inductive hypothesis.  Therefore, the sequence $(s_p,\ldots,s_1)$ is unique.
\end{proof}

\section{Macaulay Representations of Ideals and Quotients}

We begin by recalling a standard fact about lex segments, which allows us to study them using induction.

\begin{theorem}\label{t:lexdecomp}
Let $\mathcal{I}$ and $\mathcal{Q}$ be initial and terminal lex segments, respectively, inside $\M{n}{\delta}$.  Then:
  \begin{enumerate}
  \item Suppose $\beta\ge1$ is minimal such that $x_1^\beta x_n^{\delta-\beta}\in\mathcal{I}$.  Then $\mathcal{I}=x_1^\beta(\M{n}{\delta-\beta})\oplus x_1^{\beta-1}\mathcal{J}$, for some initial lex segment $\mathcal{J}\subseteq\M{[2,n]}{\delta-\beta+1}$.
  \item Suppose $\gamma\ge2$ is minimal such that $x_\gamma^\delta\in\mathcal{Q}$.  Then $\mathcal{Q}=x_{\gamma-1}\mathcal{R}\oplus\M{[\gamma,n]}{\delta}$, for some terminal lex segment $\mathcal{R}\subseteq\M{[\gamma-1,n]}{\delta-1}$.
  \end{enumerate}
\end{theorem}

\begin{remark}
Theorem \ref{t:lexdecomp} does not require that $\mathcal{I}$ (or $\mathcal{Q}$) be nontrivial or proper.  In the extreme cases, we simply let $\mathcal{J}$ (or $\mathcal{R}$) be the trivial vector space.
\end{remark}

We will prove a stronger version of Theorem \ref{t:lexdecomp}:

\begin{proposition}\label{p:lexdecompspec}
Let $\mathcal{I}=\I{n}{\delta}{m}$ and $\mathcal{Q}=\Q{n}{\delta}{m}$ be initial and terminal lex segments corresponding to $m$.  Then:
  \begin{enumerate}
  \item Suppose $\beta\ge1$ is minimal such that $x_1^\beta x_n^{\delta-\beta}\in\mathcal{I}$.  Then $\mathcal{I}=x_1^\beta(\M{n}{\delta-\beta})\oplus x_1^{\beta-1}\I{[2,n]}{\delta-\beta+1}{\frac{m}{x_1^{\beta-1}}}$.  In the case where there is no such minimal value, i.e. $m=x_1^\delta$ and $\mathcal{I}$ is trivial, we may take $\beta=\delta+1$.
  \item Suppose $\gamma\ge2$ is minimal such that $x_\gamma^\delta\in\mathcal{Q}$.  Then $\mathcal{Q}=x_{\gamma-1}\Q{[\gamma-1,n]}{\delta-1}{\frac{m}{x_{\gamma-1}}}\oplus\M{[\gamma,n]}{\delta}$.  In the case where there is no such minimal value, i.e. $m=x_n^\delta$ and $\mathcal{Q}$ is trivial, we may take $\gamma=n+1$.
  \end{enumerate}
\end{proposition}
\begin{proof}\leavevmode
\begin{enumerate}
    \item[(1)] Note that $x_1^{\beta-1}\,\big|\,m$ but $x_1^\beta\nmid m$.  We will treat each monomial space as a $k$-vector space spanned by its generating monomials.  Define $\phi:\mathcal{I}\to\mathcal{I}$, as a $k$-vector space map, to map $\mu\mapsto0$ if $x_1^\beta\,\big|\,\mu$ and $\mu\mapsto\mu$ otherwise.  We claim that $\ker(\phi)=x_1^\beta\M{n}{\delta-\beta}$ and $\im(\phi)=x_1^{\beta-1}\I{[2,n]}{\delta-\beta+1}{\frac{m}{x_1^{\beta-1}}}$.  We immediately have $\ker(\phi)\subseteq x_1^\beta\M{n}{\delta-\beta}$ and $x_1^{\beta-1}\I{[2,n]}{\delta-\beta+1}{\frac{m}{x_1^{\beta-1}}}\subseteq\im(\phi)$.
    
    For any monomial $m'\in\M{n}{\delta-\beta}$, we also have $x_1^\beta m'\lexl m$ since $x_1^\beta\nmid m$.  Thus, $x_1^\beta\M{n}{\delta-\beta}\subseteq\ker(\phi)$.  We conclude that $\ker(\phi)=x_1^\beta\M{n}{\delta-\beta}$.
    
    Now suppose $m'\in\im(\phi)$ is nonzero, so $m'\lexl m$.  By minimality of $\beta$, $x_1^{\beta-1}\,\big|\,m'$ so that $\frac{m'}{x_1^{\beta-1}}\lexl\frac{m}{x_1^{\beta-1}}$.  Thus, $\im(\phi)\subseteq x_1^{\beta-1}\I{[2,n]}{\delta-\beta+1}{\frac{m}{x_1^{\beta-1}}}$ and we have equality.
    
    Since $\phi$ is a projection map, $\mathcal{I}=\ker(\phi)\oplus\im(\phi)=x_1^\beta\M{n}{\delta-\beta}\oplus x_1^{\beta-1}\I{[2,n]}{\delta-\beta+1}{\frac{m}{x_1^{\beta-1}}}$.

    \item[(2)] Note that $x_{\gamma-1}\,\big|\,m$ but $x_i\nmid m$ for any $1\le i\le\gamma-2$.  Define a map $\psi:\mathcal{Q}\to\mathcal{Q}$ between $k$-vector spaces by mapping $\mu\mapsto0$ if $x_{\gamma-1}\,\big|\,\mu$ and $\mu\mapsto\mu$ otherwise.   We claim $\ker(\psi)=x_{\gamma-1}\Q{[\gamma-1,n]}{\delta-1}{\frac{m}{x_{\gamma-1}}}$ and $\im(\psi)=\M{[\gamma,n]}{\delta}$.  We immediately have $x_{\gamma-1}\Q{[\gamma-1,n]}{\delta-1}{\frac{m}{x_{\gamma-1}}}\subseteq\ker(\psi)$ and $\M{[\gamma,n]}{\delta}\subseteq\im(\psi)$.
    
    Suppose $\mu\in\im(\psi)$.  Then by construction, $x_i\nmid\mu$ for all $1\le i\le\gamma-1$.  Thus, $\mu\in\M{[\gamma,n]}{\delta}$.  We conclude that $\im(\psi)=\M{[\gamma,n]}{\delta}$.
    
    Now suppose $\mu\in\ker(\psi)$.  Then $\mu=x_{\gamma-1}\mu'$ for some $\mu'\in\M{[\gamma-1,n]}{\delta-1}$.  Since $m\lexl\mu$, $\frac{m}{x_{\gamma-1}}\lexl\mu'$, so $\mu'\in\Q{[\gamma-1,n]}{\delta-1}{\frac{m}{x_{\gamma-1}}}$.  Thus, $\mu\in x_{\gamma-1}\Q{[\gamma-1,n]}{\delta-1}{\frac{m}{x_{\gamma-1}}}$.  We conclude that $\ker(\psi)=x_{\gamma-1}\Q{[\gamma-1,n]}{\delta-1}{\frac{m}{x_{\gamma-1}}}$.  Since $\psi$ is a projection map, $\mathcal{Q}=x_{\gamma-1}\Q{[\gamma-1,n]}{\delta-1}{\frac{m}{x_{\gamma-1}}}\oplus\M{[\gamma,n]}{\delta}$.\qedhere
\end{enumerate}
\end{proof}

We now induct on $n$ (or $\delta$) to express an arbitrary lex segment as a sum of monomial spaces.

\begin{theorem}\label{t:longdecomp}
Let $m$ be a degree-$\delta$ monomial in $n$ variables.  Then $\I{n}{\delta}{m}\cong\bigoplus_{i=1}^{n-1}\M{[i,n]}{\beta_i}$ and $\Q{n}{\delta}{m}\cong\bigoplus_{i=1}^{\delta}\M{[\gamma_i,n]}{i}$ where $\beta_i\ge-1,\gamma_i\ge2$ are some integers for each $i$.
\end{theorem}
\begin{proof}
We induct on $n$ for the first isomorphism and $\delta$ for the second.  Suppose the first isomorphism holds for $n-1$ variables; we will assume $\I{[2,n]}{\rho}{\mu}\cong\bigoplus_{i=2}^{n-1}\M{[i,n]}{\beta_i}$ for some values $\beta_i$, for any $\rho\in\Z^+$ and $\mu\in\M{[2,n]}{\rho}$.  Let $\beta_1=\delta-\beta$, where $\beta$ is constructed for $\I{n}{\delta}{m}$ as in Proposition \ref{p:lexdecompspec}.  By Proposition \ref{p:lexdecompspec}, $\I{n}{\delta}{m}\cong\M{n}{\beta_1}\oplus\I{[2,n]}{\beta_1+1}{\frac{m}{x_1^{\beta-1}}}$, and by the inductive hypothesis, $\I{n}{\delta}{m}\cong\M{[1,n]}{\beta_1}\oplus\left(\bigoplus_{i=2}^{n-1}\M{[i,n]}{\beta_i}\right)=\bigoplus_{i=1}^{n-1}\M{[i,n]}{\beta_i}$.

Now suppose the second isomorphism holds for monomials of degree $\delta-1$; we will assume $\Q{r}{\delta-1}{\mu}\cong\bigoplus_{i=1}^{\delta-1}\M{[\gamma_i,r]}{i}$ for some values $\gamma_i$, for all $r\in\Z^+$ and $\mu\in\M{r}{\delta-1}$.  Let $\gamma_\delta=\gamma$, where $\gamma$ is constructed for $\Q{n}{\delta}{m}$ as in Proposition \ref{p:lexdecompspec}.  By Proposition \ref{p:lexdecompspec}, $\Q{n}{\delta}{m}\cong\M{[\gamma_\delta,n]}{\delta}\oplus\Q{[\gamma_\delta-1,n]}{\delta-1}{\frac{m}{x_{\gamma-1}}}$, and by the inductive hypothesis, $\Q{n}{\delta}{m}\cong\M{[\gamma_\delta,n]}{\delta}\oplus\left(\bigoplus_{i=1}^{\delta-1}\M{[\gamma_i,n]}{i}\right)=\bigoplus_{i=1}^\delta\M{[\gamma_i,n]}{i}$.
\end{proof}

\begin{corollary}\label{c:decompdimen}
Let $\beta_i,\gamma_i$ be as in Theorem \ref{t:longdecomp}.  The $\ith{(n-1)}$ Macaulay representation of $\dim\I{n}{\delta}{m}$ is
\[\dim\I{n}{\delta}{m}=\sum_{i=1}^{n-1}\binom{\beta_i+n-i}{n-i}=\sum_{i=1}^{n-1}\binom{\beta_{n-i}+i}{i}\]
and the $\ith{\delta}$ Macaulay representation of $\dim\Q{n}{\delta}{m}$ is
\[\dim\Q{n}{\delta}{m}=\sum_{i=1}^{\delta}\binom{n-\gamma_i+i}{i}.\]
Thus, the $\ith{(n-1)}$ and $\ith{\delta}$ Macaulay coefficients for $\dim\I{n}{\delta}{m}$ and $\dim\Q{n}{\delta}{m}$ are, respectively, $s_i=\beta_{n-i}+i$ for $1\le i\le n-1$ and $t_i=n-\gamma_i+i$ for $1\le i\le\delta$.
\end{corollary}

We now prove the existence half of Theorem \ref{t:macrep}, in two ways, as a corollary to Theorem \ref{t:longdecomp}.

\begin{proof}[Proof of Theorem \ref{t:macrep} using ideals.]
  Choose $\delta$ such that $\dim\M{p+1}{\delta}>s$, and let $\mathcal{I}=\I{p+1}{\delta}{m}$ be the ideal segment of a monomial $m$ chosen so that $\mathcal{I}$ has dimension $s$.  Then by Corollary \ref{c:decompdimen},

  \[s=\dim\mathcal{I}=\sum_{i=1}^p\binom{\beta_{p-i+1}+i}{i}\]

  \noindent for some nonincreasing values $\beta_i\ge0$.  Thus, the sequence $\beta_{p-i+1}+i$ decreases as $i$ decreases, so $(\beta_1+p,\beta_2+p-1,\ldots,\beta_p+1)$ are the $\ith{p}$ Macaulay coefficients of $s$.
\end{proof}

\begin{proof}[Proof of Theorem \ref{t:macrep} using quotients.]
  Choose $n$ such that $\dim\M{n}{p}>s$, and let $\mathcal{Q}=\Q{n}{p}{m}$ be the quotient segment of a monomial $m$ chosen so that $\mathcal{Q}$ has dimension $s$.  Then by Corollary \ref{c:decompdimen},

  \[s=\dim\mathcal{Q}=\sum_{i=1}^p\binom{n-\gamma_i+i}{i}\]

  \noindent for some nonincreasing values $\gamma_i\ge2$.  Thus, the sequence $n-\gamma_i+i$ decreases as $i$ decreases, so $(n-\gamma_p+p,n-\gamma_{p-1}+p-1,\ldots,n-\gamma_1+1)$ are the $\ith{p}$ Macaulay coefficients of $s$.
\end{proof}

We are now ready to state the numerical rules characterizing Hilbert functions.

\begin{corollary}\label{c:longdecompmult} Let $L$ be a lex ideal generated in degree $\delta$, and suppose $\mathcal{I}=L_\delta$ and $\mathcal{Q}=(S/L)_\delta$.  Write
  \[\mathcal{I}\cong\bigoplus_{i=1}^{n-1}\M{[i,n]}{\beta_i}\]
  and
  \[\mathcal{Q}\cong\bigoplus_{i=1}^{\delta}\M{[\gamma_i,n]}{i}\]
  as in Theorem \ref{t:longdecomp}.
  Then
  \begin{enumerate}
  \item $S_1L_\delta\cong\bigoplus_{i=1}^{n-1}\M{[i,n]}{\beta_i+1}$.
  \item $(S/L)_{\delta+1}\cong\bigoplus_{i=1}^{\delta}\M{[\gamma_i,n]}{i+1}$
  \end{enumerate}
\end{corollary}
\begin{proof}\leavevmode
\begin{enumerate}
\item[(1)] We will induct on $n$.  By Proposition \ref{p:lexdecompspec}, $\mathcal{I}=x_1^\beta(\M{n}{\delta-\beta})\oplus x_1^{\beta-1}\I{[2,n]}{\delta-\beta+1}{\mu}$ for some monomial $\mu$, where $\beta$ is minimally chosen so that $x_1^\beta x_n^{\delta-\beta}\in\mathcal{I}$.  Then

\begin{align*}
(S_1L)_\delta&=x_1^\beta\M{n}{\delta-\beta+1}+S_1(x_1^{\beta-1}\I{[2,n]}{\delta-\beta+1}{\mu}) \\
&=x_1^\beta\M{n}{\delta-\beta+1}+x_1(x_1^{\beta-1}\I{[2,n]}{\delta-\beta+1}{\mu})+x_1^{\beta-1}T_1\I{[2,n]}{\delta-\beta+1}{\mu},
\end{align*}

where $T_1=\M{[2,n]}{1}$.  Observe that $x_1^\beta\I{[2,n]}{\delta-\beta+1}{\mu}\subseteq x_1^\beta\M{n}{\delta-\beta+1}$, so $(S_1L)_\delta=x_1^\beta\M{n}{\delta-\beta+1}\oplus x_1^{\beta-1}T_1\I{[2,n]}{\delta-\beta+1}{\mu}$.  By induction, this is isomorphic to $\M{[1,n]}{\beta_1+1}\oplus\left(\bigoplus_{i=2}^{n-1}\M{[i,n]}{\beta_i+1}\right)=\bigoplus_{i=1}^{n-1}\M{[i,n]}{\beta_i+1}$ where $\beta_1=\delta-\beta$.  This proves the result.

\item[(2)] We will induct on $\delta$.  By Proposition \ref{p:lexdecompspec}, $\mathcal{Q}=x_{\gamma-1}\Q{[\gamma-1,n]}{\delta-1}{\mu}\oplus\M{[\gamma,n]}{\delta}$ for some monomial $\mu$, where $\gamma$ is minimally chosen so that $x_\gamma^\delta\in\mathcal{Q}$.  Then $(S/L)_{\delta+1}$ is the direct sum of the spaces

\begin{align*}
S_1(x_{\gamma-1}\Q{[\gamma-1,n]}{\delta-1}{\mu})&=(x_1,x_2,\ldots,x_{\gamma-2})x_{\gamma-1}\Q{[\gamma-1,n]}{\delta-1}{\mu} \\
&\qquad\qquad+(x_{\gamma-1},\ldots,x_n)x_{\gamma-1}\Q{[\gamma-1,n]}{\delta-1}{\mu} \\
&=\{0\}+x_{\gamma-1}\Q{[\gamma-1,n]}{\delta}{\mu x_n},
\end{align*}
by Proposition \ref{p:quotientmult}, and
\begin{align*}
S_1(\M{[\gamma,n]}{\delta})&=(x_1,x_2,\ldots,x_{\gamma-2})\M{[\gamma,n]}{\delta}+x_{\gamma-1}\M{[\gamma,n]}{\delta}+(x_\gamma,\ldots,x_n)\M{[\gamma,n]}{\delta} \\
&=\{0\}+x_{\gamma-1}\M{[\gamma,n]}{\delta}+\M{[\gamma,n]}{\delta+1},
\end{align*}

where if $\gamma=2$, the space $(x_1,\ldots,x_{\gamma-2})=\{0\}$.  Note that monomials with variables whose indices precede $\gamma-1$ are in $L$.

Thus, $(S/L)_{\delta+1}=x_{\gamma-1}\Q{[\gamma-1,n]}{\delta}{\mu x_n}+x_{\gamma-1}\M{[\gamma,n]}{\delta}+\M{[\gamma,n]}{\delta+1}$.  We claim $x_{\gamma-1}\M{[\gamma,n]}{\delta}\subseteq x_{\gamma-1}\Q{[\gamma-1,n]}{\delta}{\mu x_n}$.

Let $w\in x_{\gamma-1}\M{[\gamma,n]}{\delta}$ and write $w=x_{\gamma-1}u$, where $u\in\M{[\gamma,n]}{\delta}$.  If $\mu x_n\lexl u$, then $w\in x_{\gamma-1}\Q{[\gamma-1,n]}{\delta}{\mu x_n}$.  Now suppose $u\lexleq\mu x_n$ so that $w\lexleq(x_{\gamma-1}\mu)x_n$.  Since $(S/L)_\delta=\Q{[\gamma-1,n]}{\delta}{x_{\gamma-1}\mu}$, $x_{\gamma-1}\mu\in L$ so that $w\in S_1L$.  The inclusion thus holds and we have

\begin{align*}
(S/L)_{\delta+1}&=(x_{\gamma-1},\ldots,x_n)x_{\gamma-1}\Q{[\gamma-1,n]}{\delta-1}{\mu}\oplus\M{[\gamma,n]}{\delta+1} \\
&\cong\left(\bigoplus_{i=1}^{\delta-1}\M{\gamma_i}{i+1}\right)\oplus\M{[\gamma_\delta,n]}{\delta+1} \\
&=\bigoplus_{i=1}^\delta\M{[\gamma_i,n]}{i+1},
\end{align*}

by the inductive hypothesis, where $\gamma_\delta=\gamma$.  This proves the result.\qedhere
\end{enumerate}
\end{proof}
  
\begin{corollary}[Macaulay's theorem, numerical version]
  Suppose $\mathcal{I}=L_\delta$ and $\mathcal{Q}=(S/L)_\delta$ as in Corollary \ref{c:longdecompmult}.  Then
  \[\dim S_1L_\delta=\sum_{i=1}^{n-1}\binom{\beta_{n-i}+i+1}{i}\]
  and
  \[\dim(S/L)_{\delta+1}=\sum_{i=1}^{\delta}\binom{n-\gamma_i+i+1}{i+1}.\]
\end{corollary}

\section{Macaulay Coefficients and Exponent Vectors}

By the proof of Theorem \ref{t:macrep}, the Macaulay coefficients of an integer can be viewed as the Macaulay coefficients of an ideal (or quotient) segment corresponding to a certain monomial $m$.  Thus, we should be able to determine the coefficients directly from the exponent vector of $m$.  We begin by illustrating how to do so for the ideal segment.

The process will involve taking successively smaller pieces of the factorization of $m$, which we refer to as the coarse and fine tails of $m$ for the ideal and quotient, respectively.

\begin{definition}
Let $m=x_1^{\alpha_1}\cdots x_n^{\alpha_n}=x_{j_1}\cdots x_{j_\delta}$ be a monomial.  For $0\le i\le n-1$, define the \emph{$\ith{i}$ coarse tail of} $m$ to be $\ct_i(m):=x_{i+1}^{\alpha_{i+1}}x_{i+2}^{\alpha_{i+2}}\cdots x_n^{\alpha_n}$.  For $0\le i\le\delta-1$, define the \emph{$\ith{i}$ fine tail of} $m$ to be $\ft_i(m):=x_{j_{i+1}}x_{j_{i+2}}\cdots x_{j_\delta}$.
\end{definition}

\begin{example}
In $k[a,b,c,d]$, let $m=a^2bd^4$.  Then $\ct_0(m)=a^2bd^4$, $\ct_1(m)=bd^4$, and $\ct_2(m)=\ct_3(m)=d^4$.  Also, $\ft_0(m)=a^2bd^4$, $\ft_1(m)=abd^4$, $\ft_2(m)=bd^4$, $\ft_3(m)=d^4$, $\ft_4(m)=d^3$, $\ft_5(m)=d^2$, and $\ft_6(m)=d$.
\end{example}

The terms in the Macaulay representation come directly from the summands in the decomposition in Theorem \ref{t:longdecomp}.  We will show how the tails determine the summands.

\begin{example}\label{e:idrep}
Consider the monomial $m_1=a^2bd^3\!f^2$ in $k[a,b,c,d,e,f]$, so $n=6$ and $\delta=8$.  We want to find $\dim\I{6}{8}{m_1}$ by first understanding the structure of $\I{6}{8}{m_1}$.

Applying Proposition \ref{p:lexdecompspec}, we obtain

\[\I{6}{8}{a^2bd^3\!f^2}=a^3\M{6}{5}\oplus a^2\I{[2,6]}{6}{bd^3\!f^2},\]

\noindent and $bd^3\!f^2=\ct_1(a^2bd^3\!f^2)=\frac{m_1}{a^2}$.  Let $m_2=bd^3\!f^2$.  Obtaining another ideal segment, we may perform the same action on $\I{[2,6]}{6}{m_2}$:

\[\I{[2,6]}{6}{bd^3\!f^2}=b^2\M{[2,6]}{4}\oplus b\I{[3,6]}{5}{d^3\!f^2},\]

\noindent and $d^3\!f^2=\ct_2(a^2bd^3\!f^2)=\frac{m_2}{b}$.  Let $m_3=d^3\!f^2$.  Once again, we have

\[\I{[3,6]}{5}{d^3\!f^2}=c^1\M{[3,6]}{4}\oplus c^0\I{[4,6]}{5}{d^3\!f^2}.\]

Observe that $m_4=d^3\!f^2$ is equal to $m_3$, but these monomials reside in different spaces, and hence have ideal segments of different lengths.  As seen in the previous step, the ideal segments we construct may not necessarily decrease in degree, but they will be over a strictly decreasing number of variables.  Expanding the overall expression at this step, we have

\[\I{6}{8}{a^2bd^3\!f^2}=a^3\M{6}{5}\oplus a^2b^2\M{[2,6]}{4}\oplus a^2b^1c^1\M{[3,6]}{4}\oplus a^2b^1c^0\I{[4,6]}{5}{d^3\!f^2}.\]

We find that $\dim\I{6}{8}{m_1}=\dim\M{6}{5}+\dim\M{[2,6]}{4}+\dim\M{[3,6]}{4}+\dim\I{[4,6]}{5}{m_4}$.  At each step, a new term is added to this calculation, and at the end this should result in the sum of dimensions of monomial spaces, which are individually much easier to find than for lex spaces.  The full decomposition is

\[\I{6}{8}{a^2bd^3\!f^2}=a^3\M{6}{5}\oplus a^2b^2\M{[2,6]}{4}\oplus a^2bc\M{[3,6]}{4}\oplus a^2bd^4\M{[4,6]}{1}\oplus a^2bd^3e^1\M{[5,6]}{1},\]

\noindent so the final sum is

\begin{align*}
\dim\I{6}{8}{a^2bd^3\!f^2}
&=\dim\M{6}{5}+\dim\M{[2,6]}{4}+\dim\M{[3,6]}{4}+\dim\M{[4,6]}{1}+\dim\M{[5,6]}{1} \\[5pt]
&=\binom{6+5-1}{6-1}+\binom{5+4-1}{5-1}+\binom{4+4-1}{4-1}+\binom{3+1-1}{3-1}+\binom{2+1-1}{2-1} \\[5pt]
&=\binom{10}{5}+\binom{8}{4}+\binom{7}{3}+\binom{3}{2}+\binom{2}{1} \\[5pt]
&=252+70+35+3+2 \\[5pt]
&=362.
\end{align*}

Thus, there are $362$ degree-$8$ monomials in $6$ variables preceding $a^2bd^3\!f^2$ in the lex order.
\end{example}

\begin{lemma}\label{l:idealrep}
Let $n$ be given and $m_1=x_1^{\alpha_1}\cdots x_n^{\alpha_n}$.  Then
\[\dim\I{n}{\delta}{m_1}=\sum_{i=1}^{n-1}\binom{i+\deg(\ct_{n-i}(m_1))-1}{i}.\]
\end{lemma}
\begin{proof}
We induct on $n$.  We have $x_1^{\alpha_1+1}x_n^{\delta-\alpha_1-1}\lexl m_1$, but $m_1\lexleq x_1^{\alpha_1}x_n^{\delta-\alpha_1}$.  By Proposition \ref{p:lexdecompspec},

\[\I{n}{\delta}{m_1}=x_1^{\alpha_1+1}\M{[1,n]}{\delta-\alpha_1-1}\oplus x_1^{\alpha_1}\I{[2,n]}{\delta-\alpha_1}{m_2},\]

\noindent where $m_2=\ct_1(m_1)=\frac{m_1}{x_1^{\alpha_1}}$.  Thus,

\begin{align*}
\dim\I{n}{\delta}{m_1}
&=\dim\M{[1,n]}{\delta-\alpha_1-1}+\dim\I{[2,n]}{\delta-\alpha_1}{m_2} \\
&=\binom{\delta-\alpha_1-1+n-1}{n-1}+\sum_{i=1}^{n-2}\binom{i+\deg(\ct_{(n-1)-i}^T(m_2))-1}{i},
\end{align*}

\noindent by the inductive hypothesis inside $T=k[x_2,\ldots,x_n]$, where $\ct_j^T(m_2)$ is the $\ith{j}$ coarse tail of $m_2$ in $T$, i.e. obtained by deleting all $x_\ell$ with $2\le\ell\le j+1$.  This can be viewed as a coarse tail in $k[x_1,\ldots,x_n]$ by observing $\ct_j^T(m_2)=\ct_{j+1}(m_2)=\ct_{j+1}(m_1)$.  Thus,

\begin{align*}
\dim\I{n}{\delta}{m_1}
&=\binom{\delta-\alpha_1-1+n-1}{n-1}+\sum_{i=1}^{n-2}\binom{i+\deg(\ct_{(n-1)-i}^T(m_2))-1}{i} \\
&=\binom{\deg(\ct_1(m_1))-1+n-1}{n-1}+\sum_{i=1}^{n-2}\binom{i+\deg(\ct_{n-i}(m_1))-1}{i} \\
&=\sum_{i=1}^{n-1}\binom{i+\deg(\ct_{n-i}(m_1))-1}{i},
\end{align*}

\noindent as desired.
\end{proof}

Lemma \ref{l:idealrep} is the Macaulay representation of $\dim\I{n}{\delta}{m_1}$, which proves the following:

\begin{theorem}\label{t:idealcoeff}
Let $m$ be a degree-$\delta$ monomial in $n$ variables.  The $\ith{(n-1)}$-Macaulay coefficients of $\dim\I{n}{\delta}{m}=\sum_i\binom{s_i}{i}$ are $s_i=i+\deg(\ct_{n-i}(m))-1$ for each $i$.
\end{theorem}

\begin{definition}
We will call the sum $\sum_{i=1}^{n-1}\binom{s_i}{i}$ in Theorem \ref{t:idealcoeff} the \emph{$\ith{n}$ ideal representation of $m$} and the corresponding integers $(s_{n-1},s_{n-2},\ldots,s_1)$ the \emph{$\ith{n}$ ideal coefficients of $m$}.
\end{definition}

Note that the ideal cofficients do not depend on the degree of $m$; in fact, the set of monomials with the same ideal coefficients as $m$ is $\{x_1^im\ |\ i\ge-\alpha_1\}$.

We can now determine the $\ith{n}$ ideal coefficients for a monomial given its exponent vector.

\begin{example}\label{e:idcoeffcomp}
We computed the $\ith{6}$ ideal coefficients $(s_i)=(10,8,7,3,2)$ for $m=a^2bd^3\!f^2$ in $k[a,b,c,d,e,f]$ by converting its ideal segment into a direct sum of monomial spaces and finding the dimension of each as a $k$-vector space.  We can alternatively use Theorem \ref{t:idealcoeff} to find these coefficients directly:
\begin{align*}
s_1&=1+\deg(f^2)-1             & &=2 \\
s_2&=2+\deg(e^0f^2)-1          & &=3 \\
s_3&=3+\deg(d^3e^0f^2)-1       & &=7 \\
s_4&=4+\deg(c^0d^3e^0f^2)-1    & &=8 \\
s_5&=5+\deg(b^1c^0d^3e^0f^2)-1 & &=10.
\end{align*}
\end{example}

\begin{example}\label{e:quotrep}
In Example \ref{e:idrep}, we found the dimension of the ideal segment of the monomial $m_1=a^2bd^3\!f^2$ in $k[a,b,c,d,e,f]$.  Now we will find the dimension of $\Q{6}{8}{m_1}$, which is spanned by the degree-$8$ monomials in the variables $a$ through $f$ strictly lex-smaller than $m_1$.  Note that $m_1$ factors as $m_1=aabdddf\!\!f$.  Writing $m_1=x_{j_1}x_{j_2}\cdots x_{j_\delta}$, we have $j_1=j_2=1$, $j_3=2$, $j_4=j_5=j_6=4$, and $j_7=j_8=6$.  By Proposition \ref{p:lexdecompspec},

\[\Q{6}{8}{a^2bd^3\!f^2}=a\Q{[1,6]}{7}{abd^3\!f^2}\oplus\M{[2,6]}{8}.\]

Every monomial not divisible by $a$, namely those in the variables $b$ through $f$, is in $\Q{6}{8}{m_1}$.  The others, after factoring out $a$, produce another quotient segment $\Q{[1,6]}{7}{m_2}$ where $m_2=abd^3\!f^2=\ft_1(a^2bd^3\!f^2)=\frac{m_1}{a}$.  As before, the generators of $\Q{[1,6]}{7}{m_2}$ contain all monomials in the variables $b$ through $f$, but this time of one degree lower.  The rest form a quotient segment whose generators have one more factor of $a$:

\[\Q{[1,6]}{7}{abd^3\!f^2}=a\Q{[1,6]}{6}{bd^3\!f^2}\oplus\M{[2,6]}{7}=a\Q{[2,6]}{6}{bd^3\!f^2}\oplus\M{[2,6]}{7},\]

\noindent with $bd^3\!f^2=\ft_2(m_1)=\frac{m_2}{a}$.

In Example \ref{e:idrep}, each step reduced the number of variables exactly once whereas the degree decreased according to the coarse tails of the monomial.  For the quotient segment, the opposite happens: the degree decreases exactly once per step while the number of variables decreases according to the fine tails.

We may choose to view the new quotient segment $\Q{[1,6]}{6}{bd^3\!f^2}$ as over the variables $b$ through $f$ instead of $a$ through $f$ by Lemma \ref{l:reducequoseg}.  This will help keep track of which variables are ``removed'' as we obtain quotient segments of decreasing degree.  More precisely, the left endpoint of the bracket after the $\ith{i}$ step is equal to $j_{i+1}$.

Continuing the decomposition inside the smaller ring, we get

\[\Q{[2,6]}{6}{bd^3\!f^2}=b\Q{[4,6]}{5}{d^3\!f^2}\oplus \M{[3,6]}{6}.\]

Expanding the overall expression so far, we have

\[\Q{6}{8}{a^2bd^3\!f^2}=a^2b\Q{[4,6]}{5}{d^3\!f^2}\oplus a^2\M{[3,6]}{6}\oplus a\M{[2,6]}{7}\oplus\M{[2,6]}{8}.\]

The complete decomposition will be

\begin{align*}
\Q{6}{8}{a^2bd^3\!f^2}&=\M{[2,6]}{8}\oplus a\M{[2,6]}{7}\oplus a^2\M{[3,6]}{6}\oplus a^2b\M{[5,6]}{5} \\[5pt]
&\qquad\quad{}\oplus a^2bd\M{[5,6]}{4}\oplus a^2bd^2\M{[5,6]}{3}\oplus a^2bd^3\M{0}{2}\oplus a^2bd^3f\M{0}{1}.
\end{align*}

Taking the dimension gives us the final count:

\begin{align*}
\dim\Q{6}{8}{a^2bd^3\!f^2}
&=\dim\M{[2,6]}{8}+\dim\M{[2,6]}{7}+\dim\M{[3,6]}{6}+\dim\M{[5,6]}{5} \\[5pt]
&\qquad\quad{}+\dim\M{[5,6]}{4}+\dim\M{[5,6]}{3}+\dim\M{0}{2}+\dim\M{0}{1} \\[5pt]
&=\binom{5+8-1}{8}+\binom{5+7-1}{7}+\binom{4+6-1}{6}+\binom{2+5-1}{5} \\[5pt]
&\qquad\quad{}+\binom{2+4-1}{4}+\binom{2+3-1}{3}+\binom{0+2-1}{2}+\binom{0+1-1}{1} \\[5pt]
&=\binom{12}{8}+\binom{11}{7}+\binom{9}{6}+\binom{6}{5}+\binom{5}{4}+\binom{4}{3}+\binom{1}{2}+\binom{0}{1} \\[5pt]
&=495+330+84+6+5+4+0+0 \\[5pt]
&=924.
\end{align*}

Thus, there are $924$ degree-$8$ monomials in $6$ variables following $m_1$ in the lex order.  We may verify this result along with the one obtained in Example \ref{e:idrep} by combining all monomials strictly lex-larger than $m_1$ with all monomials strictly lex-smaller than it, as well as $m_1$ itself, to have $362+924+1=1287=\binom{6+8-1}{8}$, the total number of monomials of degree-$8$ in $6$ variables.
\end{example}

\begin{lemma}\label{l:quotrep}
Let $\delta$ be a positive integer and $m_1=x_1^{\alpha_1}\cdots x_n^{\alpha_n}$.  Then
\[\dim\Q{n}{\delta}{m_1}=\sum_{i=1}^\delta\binom{n-\min(\ft_{\delta-i}(m_1))+i-1}{i}.\]
\end{lemma}
\begin{proof}
We induct on $\delta$.  We have $m_1\lexl x_{j_1+1}^\delta$, but $x_{j_1}^\delta\lexleq m_1$.  By Proposition \ref{p:lexdecompspec},

\[\Q{n}{\delta}{m_1}=\M{[j_1+1,n]}{\delta}\oplus x_{j_1}\Q{[j_2,n]}{\delta-1}{m_2},\]

\noindent where $m_2=\ft_1(m_1)=\frac{m_1}{x_{j_1}}$.  Thus,

\begin{align*}
\dim\Q{n}{\delta}{m_1}
&=\dim\M{[j_1+1,n]}{\delta}+\dim\Q{[j_2,n]}{\delta-1}{m_2} \\
&=\binom{\delta+(n-(j_1+1)+1)-1}{\delta}+\sum_{i=1}^{\delta-1}\binom{n-\min(\ft_{(\delta-1)-i}(m_2))+i-1}{i},
\end{align*}

\noindent by the inductive hypothesis.  Thus,

\begin{align*}
\dim\Q{n}{\delta}{m_1}
&=\binom{\delta+(n-(j_1+1)+1)-1}{\delta}+\sum_{i=1}^{\delta-1}\binom{n-\min(\ft_{(\delta-1)-i}(m_2))+i-1}{i} \\
&=\binom{n-\min(\ft_{\delta-\delta}(m_1))+\delta-1}{\delta}+\sum_{i=1}^{\delta-1}\binom{n-\min(\ft_{\delta-i}(m_1))+i-1}{i} \\
&=\sum_{i=1}^\delta\binom{n-\min(\ft_{\delta-i}(m_1))+i-1}{i},
\end{align*}

\noindent as desired.
\end{proof}

\begin{theorem}\label{t:quotcoeff}
Let $m$ be a degree-$\delta$ monomial in $n$ variables.  The $\ith{\delta}$-Macaulay coefficients of $\dim\Q{n}{\delta}{m}=\sum_i\binom{t_i}{i}$ are $t_i=n-\min(\ft_{\delta-i}(m))+i-1$ for each $i$.
\end{theorem}

\begin{definition}
We will call the sum $\sum_{i=1}^\delta\binom{t_i}{i}$ in Theorem \ref{t:quotcoeff} the \emph{$\ith{\delta}$ quotient representation of $m$} and the corresponding integers $(t_\delta,t_{\delta-1},\ldots,t_1)$ the \emph{$\ith{\delta}$ quotient coefficients of $m$}.
\end{definition}

\begin{example}
As in Example \ref{e:idcoeffcomp}, we obtain the $\ith{8}$ quotient coefficients $(t_i)=(12,11,9,6,5,4,1,0)$ of $m=a^2bd^3\!f^2$ using a method more direct than shown in Example \ref{e:quotrep}.

\begin{align*}
t_1&=6-\min(f)+1-1            & &=0 \\
t_2&=6-\min(f^2)+2-1          & &=1 \\
t_3&=6-\min(df^2)+3-1         & &=4 \\
t_4&=6-\min(d^2\!f^2)+4-1     & &=5 \\
t_5&=6-\min(d^3\!f^2)+5-1     & &=6 \\
t_6&=6-\min(bd^3\!f^2)+6-1    & &=9 \\
t_7&=6-\min(abd^3\!f^2)+7-1   & &=11 \\
t_8&=6-\min(a^2bd^3\!f^2)+8-1 & &=12 \\
\end{align*}
\end{example}

The language of coarse and fine tails also allow us to state a stronger version of Theorem \ref{t:longdecomp}.

\begin{stheorem}\label{st:strongdecomp}
Let $m$ be a degree-$\delta$ monomial in $n$ variables.  Then

\begin{align*}
\I{n}{\delta}{m}&=\bigoplus_{i=1}^{n-1}\frac{mx_i}{\ct_{i}(m)}\M{[i,n]}{\deg(\ct_i(m))-1}\text{\quad and} \\
\Q{n}{\delta}{m}&=\bigoplus_{i=1}^{\delta}\frac{m}{\ft_{i-1}(m)}\M{[\min(\ft_{i-1}(m))+1,n]}{\delta-i+1}.
\end{align*}
\end{stheorem}
\begin{proof}
    The proofs are very similar inductions to those of Lemmas \ref{l:idealrep} and \ref{l:quotrep} up to reindexing, without taking dimensions.
\end{proof}

We are now ready to present the main result.

\begin{notation}
For a monomial in $\M{n}{\delta}$ with $\ith{n}$ ideal coefficients $s_i$ and $\ith{\delta}$ quotient coefficients $t_i$, denote $\Sc{n}{\delta}{m}=\{s_i\ |\ 1\le i\le n-1\}$ and $\Tc{n}{\delta}{m}=\{t_i\ |\ 1\le i\le\delta\}$.
\end{notation}

\begin{example}\label{e:setpart}
Consider the monomial $a^2bd^3\!f^2$ as in Examples \ref{e:idrep} and \ref{e:quotrep}.  We have that $\Sc{6}{8}{a^2bd^3\!f^2}=\{10,8,7,3,2\}$ and $\Tc{6}{8}{a^2bd^3\!f^2}=\{12,11,9,6,5,4,1,0\}$.  Observe that not only are these disjoint sets, but their union is the set of all integers between $0$ and $12$.  A similar statement is true for other monomials.  If we consider the monomial $b^2cd$ in $k[a,b,c,d]$, we have $\Sc{4}{4}{b^2cd}=\{6,3,1\}$ and $\Tc{4}{4}{b^2cd}=\{5,4,2,0\}$.
\end{example}

\begin{example}\label{e:coeffinherit}
We can observe how $\Sc{n}{\delta}{m}$ and $\Tc{n}{\delta}{m}$ are affected when we change the value of $n$ or $\delta$.  Consider the monomial $e^2\!fg$ in $k[a,b,c,d,e,f,g]$, which is obtained from $b^2cd$ by raising each variable index by $3$.  Its quotient coefficients $(5,4,2,0)$ are the same as those for $b^2cd$, but it has three additional ideal coefficients: $(9,8,7,6,3,1)$.  The monomial $a^3b^2cd$ in $k[a,b,c,d]$, obtained from $b^2cd$ by multiplying by $a^3$, has identical ideal coefficients $(6,3,1)$ to $b^2cd$, but quotient coefficients $(9,8,7,5,4,2,0)$.  In both cases, we have raised $n$ or $\delta$ by $3$, adding three more terms to the corresponding  representation while leaving the other unaffected.  The phenomenon observed in Example \ref{e:setpart} remains true for each, however.
\end{example}

Examples \ref{e:setpart} and \ref{e:coeffinherit} are not coincidental.

\begin{theorem}\label{t:setpart}
Let $n,\delta\ge1$ and $m\in\M{n}{\delta}$.  Then $\{\Sc{n}{\delta}{m},\Tc{n}{\delta}{m}\}$ forms a set partition of $\{0,1,\ldots,n+\delta-2\}$.
\end{theorem}

Before the proof of the theorem, we need a lemma involving a shifting operation on monomials.

\begin{definition}
Let $m=x_{j_1}x_{j_2}\cdots x_{j_\delta}$ in $\M{n}{\delta}$ and define $\sigma_i(m):=x_{j_1+i}x_{j_2+i}\cdots x_{j_\delta+i}$ in $\M{n+i}{\delta}$.  We will call $\sigma_i(m)$ the \emph{$\ith{i}$ shift of $m$}.
\end{definition}

\begin{lemma}\label{l:coeffinherit}
Let $m=x_1^{\alpha_1}x_2^{\alpha_2}\cdots x_n^{\alpha_n}=x_{j_1}x_{j_2}\cdots x_{j_\delta}$ be a degree-$\delta$ monomial in $k[x_1,\ldots,x_n]$.  Then
\begin{enumerate}
    \item[(1)] $\Sc{n}{\delta+1}{x_1m}=\Sc{n}{\delta}{m}$ and $\Tc{n}{\delta+1}{x_1m}=\Tc{n}{\delta}{m}\sqcup\{n+\delta-1\}$.
    \item[(2)] $\Tc{n+1}{\delta}{\sigma_1(m)}=\Tc{n}{\delta}{m}$ and $\Sc{n+1}{\delta}{\sigma_1(m)}=\Sc{n}{\delta}{m}\sqcup\{n+\delta-1\}$.
\end{enumerate}
\end{lemma}
\begin{proof}\leavevmode
\begin{enumerate}
    \item[(1)] We have $\ct_{n-i}(x_1^{\alpha_1+1}x_2^{\alpha_2}\cdots x_n^{\alpha_n})=x_{n-i+1}^{\alpha_{n-i+1}}\cdots x_n^{\alpha_n}=\ct_{n-i}(m)$ since $n-i+1\ge2$ for each $1\le i\le n-1$.  Thus, $\Sc{n}{\delta+1}{x_1m}=\Sc{n}{\delta}{m}$ by Theorem \ref{t:idealcoeff}.

    We have $\ft_{\delta+1-i}(x_1m)=x_{j_{\delta-i+2}}\cdots x_{j_\delta}=\ft_{\delta-i}(m)$ for each $1\le i\le\delta$, since $\delta+1-i\ge1$, and $\ft_{\delta+1-i}(x_1m)=x_1m$ when $i=\delta+1$.  Thus, $\Tc{n}{\delta+1}{x_1m}=\Tc{n}{\delta}{m}\cup\{n+\delta-1\}$ by Theorem \ref{t:quotcoeff}.  Since $i-\min(\ft_{\delta-i}(m))\le i-1\le\delta-1$ for all $1\le i\le\delta$, $n-\min(\ft_{\delta-i}(m))+i-1\le n+\delta-2$, so this union is disjoint.
    
    \item[(2)] We have $n+1-\min(\ft_{\delta-i}(\sigma_1(m)))=n+1-\min(x_{j_{\delta-i+1}+1}\cdots x_{j_\delta+1})=n+1-j_{\delta-i+1}-1=n-j_{\delta-i+1}=n-\min(\ft_{\delta-i}(m))$.  Thus, $\Tc{n+1}{\delta}{\sigma_1(m)}=\Tc{n}{\delta}{m}$ by Theorem \ref{t:quotcoeff}.

    We have $\deg(\ct_{n+1-i}(\sigma_1(m)))=\deg(x_{n+2-i}^{\alpha_{n+1-i}}\cdots x_{n+1}^{\alpha_n})=\deg(\ct_{n-i}(m))$ for each $1\le i\le n-1$, since $n+1-i\ge2$, and $\deg(\ct_{n+1-i}(\sigma_1(m)))=\delta$ when $i=n$.  Thus, $\Sc{n+1}{\delta}{\sigma_1(m)}=\Sc{n}{\delta}{m}\cup\{n+\delta-1\}$ by Theorem \ref{t:idealcoeff}.  Since $i+\deg(\ct_{n-i}(m))\le i+\delta\le n-1+\delta$ for all $1\le i\le n-1$, $i+\deg(\ct_{n-i}(m))-1\le n+\delta-2$, so this union is disjoint.\qedhere
\end{enumerate}
\end{proof}

\begin{proof}[Proof of Theorem \ref{t:setpart}]
Let $\mathcal{S}=\Sc{n}{\delta}{m}$ and $\mathcal{T}=\Tc{n}{\delta}{m}$.  We will induct on the quantity $n+\delta$.  Suppose the statement holds for all degree-$\delta'$ monomials in $n'$ variables such that $n'+\delta'=n+\delta-1$.  Consider the following two cases:
\begin{enumerate}
    \item[(1)] $x_1\,\big|\,m$.  Let $m'\in\M{n}{\delta-1}$ so that $x_1m'=m$ and let $\mathcal{S}',\mathcal{T}'$ respectively be the sets of $\ith{n}$ and $\ith{(\delta-1)}$ ideal and quotient coefficients $s_i',t_j'$ of $m'$.  By the inductive hypothesis, $\mathcal{S}'$ and $\mathcal{T}'$ are disjoint with $\mathcal{S}'\sqcup\mathcal{T}'=\{0,1,\ldots,n+\delta-3\}$.  By Lemma \ref{l:coeffinherit}, $\mathcal{S}=\mathcal{S}'$ and $\mathcal{T}=\mathcal{T}'\sqcup\{n+\delta-2\}$. 
 Since $n+\delta-2\notin\mathcal{S}$, $\mathcal{S}$ is disjoint from $\mathcal{T}$ with $\mathcal{S}\sqcup\mathcal{T}=\mathcal{S}'\sqcup\mathcal{T}'\sqcup\{n+\delta-2\}=\{0,1,\ldots,n+\delta-2\}$.
    \item[(2)] $x_1\nmid m$.  Let $m'\in\M{n-1}{\delta}$ so that $\sigma_1(m')=m$ and let $\mathcal{S}',\mathcal{T}'$ respectively be the sets of $\ith{(n-1)}$ and $\ith{\delta}$ ideal and quotient coefficients $s_i',t_j'$ of $m'$.  By the inductive hypothesis, $\mathcal{S}'$ and $\mathcal{T}'$ are disjoint with $\mathcal{S}'\sqcup\mathcal{T}'=\{0,1,\ldots,n+\delta-3\}$.  By Lemma \ref{l:coeffinherit}, $\mathcal{T}=\mathcal{T}'$ and $\mathcal{S}=\mathcal{S}'\sqcup\{n+\delta-2\}$.  Since $n+\delta-2\notin\mathcal{T}$, $\mathcal{S}$ is disjoint from $\mathcal{T}$ with $\mathcal{S}\sqcup\mathcal{T}=\mathcal{S}'\sqcup\mathcal{T}'\sqcup\{n+\delta-2\}=\{0,1,\ldots,n+\delta-2\}$.\qedhere
    \end{enumerate}
\end{proof}

\begin{corollary}
Let $U=\{0,\ldots,n+\delta-2\}$ and $m\in\M{n}{\delta}$.  Then $\Sc{n}{\delta}{m}=U\setminus\Tc{n}{\delta}{m}$.  In other words, if one set of coefficients is known, the other is obtained by taking its complement in $U$.
\end{corollary}

\begin{corollary}\label{c:monobijpart}
Given $p\in\Z^+$ such that $p=n+\delta-2$, with $n\ge2$ and $\delta\ge1$, and a subset $\mathcal{S}\subseteq U=\{0,\ldots,p\}$ with $|\mathcal{S}|=n-1$, there exists a unique degree-$\delta$ monomial $m$ in $n$ variables with $\Sc{n}{\delta}{m}=\mathcal{S}$ and $\Tc{n}{\delta}{m}=U\setminus\mathcal{S}$.  This demonstrates a bijective correspondence between the set of degree-$\delta$ monomials in $n$ variables and the subsets of $U$ of size $n-1$ given by $m\longleftrightarrow\Sc{n}{\delta}{m}$.
\end{corollary}

Such a monomial $m$ in Corollary \ref{c:monobijpart} may be constructed: let $\mathcal{S}\subsetneq U$ be nonempty and proper.  Let $n=|\mathcal{S}|+1$ and $\delta=p-n+2$.  Additionally, let $s_i\in\mathcal{S},t_i\in U\setminus\mathcal{S}$ be the $\ith{i}$ element in their respective sets in increasing order.  The monomial $m\in\M{n}{\delta}$ with ideal coefficients $\mathcal{S}$ and quotient coefficients $U\setminus\mathcal{S}$ is the $\ith{q}$ monomial generator of $\M{n}{\delta}$ in the lex order, where

\[q=1+\sum_{i=1}^{n-1}\binom{s_i}{i}\quad\text{or}\quad q=\binom{n+\delta-1}{\delta}-\sum_{i=1}^\delta\binom{t_i}{i}.\]

The monomial itself may be obtained using either the ideal or quotient coefficients.  First, we prove a proposition:

\begin{proposition}
Let $m=x_1^{\alpha_1}x_2^{\alpha_2}\cdots x_n^{\alpha_n}$ have $\ith{n}$ ideal coefficients $s_i$.  Then:
\begin{align*}
\alpha_1&=n+\delta-2-s_{n-1}, \\
\alpha_n&=s_1,\quad\text{and} \\
\alpha_i&=s_{n-i+1}-s_{n-i}-1
\end{align*}
for $2\le i\le n-1$.  If $m=x_{j_1}x_{j_2}\cdots x_{j_\delta}$ has $\ith{\delta}$ quotient coefficients $t_i$, then:
\begin{align*}
j_i=n-t_{\delta-i+1}+\delta-i
\end{align*}
for $1\le i\le\delta$.
\end{proposition}
\begin{proof}
For the ideal case, the first two statements follow directly from the formula in Theorem \ref{t:idealcoeff}.  Since $s_\ell=\ell+\deg(\ct_{n-\ell}(m))-1$ for $2\le\ell\le n-1$, $\deg(\ct_{n-\ell}(m))=s_\ell-\ell+1$.  Now, $\alpha_{n-\ell+1}=\deg(\ct_{n-\ell}(m))-\deg(\ct_{n-\ell+1}(m))=s_\ell-s_{\ell-1}-1$.  Setting $i=n-\ell+1$, we have $\alpha_i=s_{n-i+1}-s_{n-i}-1$ for $2\le i\le n-1$.

For the quotient case, apply Theorem \ref{t:quotcoeff} with $i=\delta-\ell+1$ and isolate $\min(\ft_{\ell-1}(m))=j_\ell$.
\end{proof}

\begin{example}
Let $n=\delta=4$ and $m=b^2cd\in k[a,b,c,d]$.  In Example \ref{e:setpart}, we computed $\Sc{4}{4}{m}=\{6,3,1\}$ and $\Tc{4}{4}{m}=\{5,4,2,0\}$.  To reconstruct $m$ using $\Sc{4}{4}{m}$, set $s_1=1$, $s_2=3$, and $s_3=6$.  Thus, $\alpha_4=s_1=1$, $\alpha_3=s_2-s_1-1=1$, and $\alpha_2=s_3-s_2-1=2$.  We have $m=a^{\alpha_1}b^2cd$ where $\alpha_1$ is chosen to make $\delta=4$, so $\alpha_1=0$.  Therefore, $m=b^2cd$.

To obtain $m$ using $\Tc{4}{4}{m}$, set $t_1=0$, $t_2=2$, $t_3=4$, and $t_4=5$.  Then
\begin{align*}
j_1&=4-t_4+4-1=2, \\
j_2&=4-t_3+4-2=2, \\
j_3&=4-t_2+4-3=3,\quad\text{and} \\
j_4&=4-t_1+4-4=4.
\end{align*}
These indices respectively correspond to the variables $b$, $b$, $c$, and $d$ within $k[a,b,c,d]$, so $m=b^2cd$.
\end{example}

\end{document}